\documentclass[preprint]{amsart}   % onecolumn (ditto)
\usepackage[utf8]{inputenc}
\usepackage[english]{babel}
\usepackage{graphicx}
\usepackage{amsmath}
\usepackage{amsfonts}
\usepackage{amssymb}
\usepackage[numbers]{natbib}
\usepackage[none]{hyphenat}
\usepackage{amssymb}
\usepackage{enumerate}
\usepackage[mathscr]{euscript}
\newtheorem{theorem}{Theorem}
\newtheorem{lemma}{Lemma}
\newtheorem{example}{Example}

\usepackage{ragged2e}

\begin{document}

\title{Usual stochastic ordering results for series and parallel systems with components having Exponentiated Chen distribution}
       % if too long for running head

\author{Madhurima Datta and Nitin Gupta\\
Indian Institute of Technology Kharagpur}

\address{Department of Mathematics, Indian Institute of Technology Kharagpur}

 \email[M. Datta]{ madhurima92.datta@iitkgp.ac.in}
 \email[N. Gupta]{ nitin.gupta@maths.iitkgp.ac.in}
 
 \maketitle
\begin{abstract}
	In this paper we have discussed the usual stochastic ordering relations between two system. Each system consists of n mutually independent components. The components follow Exponentiated (Extended) Chen distribution with three parameters $\alpha, \beta, \lambda$. Two popular system are taken into consideration, one is the series system and another is the parallel system. The results in this paper are obtained by varying one parameter and the other parameters are kept constant. The hazard rate ordering or reversed hazard rate ordering relations that are not possible for series or parallel systems have been demonstrated with the help of counterexamples.


\end{abstract}
MSC 62N05 \\

Exponentiated Chen distribution is an extension of the family of distributions obtained by \cite{1} using Lehman alternatives. It is used for modeling survival data. The family of distributions obtained by using Lehman alternatives are known as exponentiated type family. The resultant cumulative distribution function is obtained as follows
$$F(x,\alpha)=(F_{0}(x))^{\alpha}, ~ x>0, ~\alpha >0,$$
where $F_{0}(x)$ is the baseline distribution and $F(x,\alpha)$ is the generalization of $F_{0}(x)$ and $\alpha$ is a parameter. This model is also known as the Proportional reversed hazard rate (PRHR) model where the parameter $\alpha$ is a proportionality constant. The Chen distribution function introduced by \cite{1} is 
\begin{equation}\label{CD}
F(x; \beta, \lambda)=(1-e^{\lambda(1-e^{x^{\beta}})}) , ~x>0, \beta >0 , \lambda >0.
\end{equation}
Applying the transformation $T=(e^{X^{\beta}}-1)^{1/\beta}$ in equation \eqref{CD}, we observe that $T$ follows Weibull distribution with scale parameter $\lambda$ and shape parameter $\beta$. Also \cite{3} extended Chen distribution by adding a parameter and the survival function of the resultant Extended Weibull distribution is
\begin{equation}
\overline{F}(x;\alpha, \beta, \lambda)= (e^{\lambda(1-e^{{(\frac{x}{\alpha})}^{\beta}})})^{\alpha}, ~x>0, \alpha>0, \beta >0 , \lambda >0.
\end{equation}
\cite{4} introduced another shape parameter to the Extended Weibull distribution and obtained a four parameter modified Weibull extension distribution using the Marshall-Olkin technique. \par
Later \cite{2} introduced the generalization of Chen distribution given by \cite{1} by introducing a new parameter $\alpha$. The new distribution function with parameters $\alpha, \beta, \lambda$ is 
\begin{equation}
F(x;\alpha, \beta, \lambda)=(1-e^{\lambda(1-e^{x^{\beta}})})^{\alpha} , ~x>0, \alpha>0, \beta >0 , \lambda >0.
\end{equation}
\cite{5} discussed various important properties of Exponentiated Chen distribution such as the density function can be either decreasing or unimodal depending on the parameters $\alpha$ and $ \beta$. Also the hazard rate function can be bathtub shaped or increasing depending on $\alpha$ and $\beta$. The reversed hazard rate function of Exponentiated Chen distribution is
\begin{equation}
\tilde{r}(x;\alpha,\beta,\lambda)=\dfrac{\alpha\beta\lambda x^{\beta-1}e^{x^{\beta}}e^{\lambda(1-e^{x^{\beta}})}}{1-e^{\lambda(1-e^{x^{\beta}})}}, ~x>0, \alpha>0, \beta >0 , \lambda >0.
\end{equation}
In this paper we will explore the usual stochastic ordering relations for the minimum and maximum ordered statistics (series and parallel systems respectively) for two different samples whose components follows the Exponentiated Chen distribution (ECD). These order statistics are extremely important in reliability theory. Consider a set of random variables $X_1,\ldots,X_n$, they can be arranged as $X_{1:n},X_{2:n},\ldots, X_{n:n}$ such that $X_{1:n} \leq X_{2:n} \leq \ldots \leq X_{n:n}$, here $X_{k:n}$ is the k-th minimum of the set and is known as the k-th order statistic. $X_{k:n}$ corresponds to the lifetime of a (n-k+1)-out-of-n system. Details on order statistics are available in the book \cite{10}. Studies of stochastic ordering of ordered statistics are extremely popular nowadays and a wide variety of stochastic ordering results are available in the literature(\cite{11}). \cite{8} pioneered the field of stochastic ordering and developed  usual stochastic ordering results for proportional hazard rate models which implied similar results for exponential distribution, gamma distribution, Weibull distribution etc. \\
Consider the random variables $X$ and $Y$ to be absolutely continuous with distribution functions $F(x)$ and $G(x)$; survival functions as $\overline{F}(x)$ and $\overline{G}(x)$; probability density functions as $f(x)$ and $g(x)$; hazard rate functions as $r(x) = \dfrac{f(x)}{\overline{F}(x)}$ and $s(x) = \dfrac{g(x)}{\overline{G}(x)}$; reversed hazard rate functions as $\tilde{r}(x) = \dfrac{f(x)}{F(x)}$ and $\tilde{s}(x) = \dfrac{g(x)}{G(x)}$. \cite{6} contains detailed explanation of the above terms. Then the stochastic ordering between the random variables $X$ and $Y$ are described as follows (see \cite{7} for further details). 

\begin{enumerate}[(a)]

\item  $X$ is smaller than $Y$ in the usual stochastic order ($X \leq_{st} Y$) if and only if $\overline{F}(x) \leq \overline{G}(x)$ $\forall$ $x \in \mathbb{R}$.
	
\item $X$ is smaller than $Y$ in hazard rate order $(X \leq_{hr} Y)$ if
	$r(x) \geq s(x), ~x \in \mathbb{R}$.
	Equivalently, if $ \dfrac{\overline{G}(x)}{\overline{F}(x)}$ is increasing in $x$ over the union of the supports of $X$ and $Y$.
	
\item $X$ is smaller than $Y$ in reversed hazard rate order $(X \leq_{rh} Y)$ if $\tilde{r}(x) \leq \tilde{s}(x), ~x \in \mathbb{R}$.
	Equivalently, if $ \dfrac{G(x)}{F(x)}$ is increasing in $x$ over the union of the supports of $X$ and $Y$.
	
\item $X$ is smaller than $Y$ in likelihood ratio order $(X \leq_{lr} Y)$ if $\dfrac{g(x)}{f(x)}$ is increasing in $x$ over the union of the supports of $X$ and $Y$. 
\end{enumerate}
The likelihood ratio ordering implies both hazard rate and reversed hazard rate ordering which again implies the usual stochastic ordering. \\
Further two real valued vectors, $\underline{a} = (a_1, \ldots, a_n)$ and $\underline{b} = (b_1, \ldots, b_n)$, then $\underline{a}$ is majorized by $\underline{b}$   ( $\underline{a} \prec^{m} \underline{b}$ ) if
		\begin{equation}
		\sum_{i=1}^{n}a_{i:n} = \sum_{i=1}^{n}b_{i:n} ~\text{and}~ \sum_{i=1}^{k}a_{i:n} \geq \sum_{i=1}^{k}b_{i:n} ~\forall~k = 1,\ldots, n-1;
		\end{equation}
 where $a_{1:n} \leq \ldots \leq a_{n:n} ~( b_{1:n} \leq \ldots \leq b_{n:n}) $ is the increasing arrangement of $a_1, \ldots, a_n\\ (b_1, \ldots, b_n)$. In general for two matrices $A=\{a_{ij}\}_{m\times n}$ and $B=\{b_{ij}\}_{m\times n}$, A is majorized by  B ($A\prec^{m} B$) if $A=BP$, where $P=\{p_{ij}\}_{n\times n}$ is a doubly stochastic matrix (this matrix need not be unique but the existence of atleast one such matrix ensures majorization).\\
Another concept essential for understanding the paper is that a real valued function $\psi$ defined on a subset of $\mathbb{R}^n$ is \emph{Schur-convex (Schur-concave)} if
	\begin{equation}
	\underline{a} \prec^{m} \underline{b} \Rightarrow \psi(\underline{a}) ~\leq (\geq)~ \psi(\underline{b}),
	\end{equation}
	where $\underline{a} = (a_1, \ldots, a_n)$ and $\underline{b} = (b_1, \ldots, b_n)$ are two real valued vectors.
Throughout the paper, the notation $a \overset{\text{sign}}{=} b$ has been used to represent sign of $a$ is same as $b$.\\

The following lemmas have been used extensively for obtaining the results discussed in this paper.

\begin{lemma}[Theorem 3.A.4, see {\cite{9}}]
	\normalfont For an open interval $\mathbb{A} \subset \mathbb{R}$, a continuously differentiable function $\psi :\mathbb{A}^{n} \rightarrow \mathbb{R}$ is Schur-convex (Schur-concave) if and only if it is symmetric on $\mathbb{A}^{n}$ and for all $i \neq j$, let $\Delta=(a_{i} - a_{j})\left(\dfrac{\partial\psi(\underline{a})}{\partial a_{i}} - \dfrac{\partial\psi(\underline{a})}{\partial a_{j}}\right)$. Then $\Delta \geq(\leq)0.$
\end{lemma}

\begin{lemma} \label{lpsi1}\normalfont
	Let $\psi_1:(0,\infty)\times (0,1) \rightarrow (0, \infty)$ be defined as 
	\begin{equation}\label{lemeq1}
	\psi_1(\alpha,y)=\dfrac{y(1-y)^{\alpha - 1}}{1-(1-y)^{\alpha}}.
	\end{equation}
	Then
	\begin{enumerate}[(i)]
		\item $\psi_1(\alpha,y)$ increases with respect to $y $ for $0< \alpha <1$, 
		\item $\psi_1(\alpha,y)$ decreases with respect to $y $ for  $\alpha >1$.
	\end{enumerate}
	
\end{lemma}
\begin{proof}
	Differentiating \eqref{lpsi1} partially with respect to $y$ we obtain,
	$$\dfrac{\partial \psi_1(\alpha,y)}{\partial y} \overset{\text{sign}}{=} -(1-y)^{\alpha}(\alpha y +(1-y)^{\alpha}-1).$$
	Let $f_1(y)= \alpha y +(1-y)^{\alpha}-1, 0<y<1$ and $f_1(0)=0$. The derivative of $f_1(y) $ is 
	\begin{align*}
	f_1^{\prime}(y)&= \alpha(1-(1-y)^{\alpha - 1})\\
	&=g_1(y) \text{ (say)}.
	\end{align*}
	Again the derivative of $g_1(y)$ is
	$$g_1^{\prime}(y) = \alpha(\alpha -1)(1-y)^{\alpha -2}.$$
	The two possible cases are 
	\begin{enumerate}[(i)]
		\item for $0<\alpha<1$, $g_1^{\prime}(y) <0$ and $g_1(0)=0 \Rightarrow g_1(y) < 0$, i.e., $f_1^{\prime}(y) <0$. And $f_1(0)=0 \Rightarrow f_1(y) <0$ which implies $\dfrac{\partial \psi_1(\alpha,y)}{\partial y} >0$, 
		\item for $\alpha >1$, $g_1^{\prime}(y) > 0$ and $g_1(0)=0 \Rightarrow g_1(y) > 0$, i.e., $f_1^{\prime}(y) > 0$. And $f_1(0)=0 \Rightarrow f_1(y) > 0$ which implies$ \dfrac{\partial \psi_1(\alpha,y)}{\partial y} < 0$. 
	\end{enumerate}
\end{proof}

Next we shall compare two systems (parallel and series) by using the usual stochastic ordering relations. We shall henceforth understand the distribution function \\ $F(x;\alpha,\beta,\lambda)$ as the distribution function of the Exponentiated Chen distribution. Along with the results we shall present few examples and counterexamples to support the results. The first theorem shows the usual stochastic ordering between the sample minimum or between two series system when only the parameter $\lambda$ is varying and all the other parameters remain constant.
\begin{theorem}
\normalfont	Let $X_1,X_2,\ldots, X_n$ be a set of n-independent random variables such that $X_i \sim F(x;\alpha, \beta, \lambda_i)$, $\alpha > 0, \beta >0, \lambda_i >0$ for $i=1,2,\ldots, n$. Consider another set of n-independent random variables $Y_1,Y_2, \ldots, Y_n$ and the distribution function of each random variable is $F(x;\alpha, \beta, \mu_i)$ for $i=1,2,\ldots, n$. Then 
	\begin{align*}
	\underline{\lambda} \prec^{m} \underline{\mu} \Rightarrow & X_{1:n} \leq_{st} Y_{1:n} (\alpha <1)\\
	& X_{1:n} \geq_{st} Y_{1:n} (\alpha > 1), 
	\end{align*}
	where $\underline{\lambda}=(\lambda_1,\lambda_2,\ldots, \lambda_n)$ and $\underline{\mu}=(\mu_1,\mu_2,\ldots,\mu_n)$.
\end{theorem}
\begin{proof}
The survival function of the series system $X_{1:n}$ is 
\begin{equation} \label{1}
\overline{F}_{X_{1:n}}(x)= \displaystyle\prod_{k=1}^n\left[1-\left(1-e^{\lambda_k(1-e^{x^{\beta}})}\right)^{\alpha}\right].
\end{equation}
$\overline{F}_{X_{1:n}}(x)$ is symmetric with respect to the vector $\underline{\lambda}=(\lambda_1,\lambda_2,\ldots, \lambda_n)$.
Partially differentiating \eqref{1} with respect to $\lambda_i$, 
\begin{equation}
\dfrac{\partial}{\partial \lambda_i}(\overline{F}_{X_{1:n}}(x))= \overline{F}_{X_{1:n}}(x)\dfrac{\alpha(1-e^{x^{\beta}})e^{\lambda_i (1-e^{x^{\beta}})}\left(1-  e^{\lambda_i(1-e^{x^{\beta}})}\right)^{\alpha -1}}{\left(1-\left(1-e^{\lambda_i(1-e^{x^{\beta}})}\right)^{\alpha}\right)}.
\end{equation}
Assume $\lambda_i \neq \lambda_j$ for $i,j=1,2,\ldots,n$ and consider 
\begin{align*}
\Delta_1 &= (\lambda_i - \lambda_j)\left(\dfrac{\partial}{\partial \lambda_i}(\overline{F}_{X_{1:n}}(x)) - \dfrac{\partial}{\partial \lambda_j}(\overline{F}_{X_{1:n}}(x))\right)\\
&=\alpha (\lambda_i - \lambda_j)\overline{F}_{X_{1:n}}(x)(1-e^{x^{\beta}})\left(\psi_1(\alpha, y_i)-\psi_1(\alpha, y_j)\right),
\end{align*}
where $y_i = e^{\lambda_i(1-e^{x^{\beta}})}$ and the function $\psi_1(\alpha,y_i)$ follows from \eqref{lemeq1}. \\
Thus using Lemma \ref{lpsi1}, when $0<\alpha<1$, $\Delta_1 >0$ $\Rightarrow \overline{F}_{X_{1:n}}(x)$ is Schur-convex with respect to $\underline{\lambda}$. And for  $\alpha >1 , \Delta_1 <0 \Rightarrow \overline{F}_{X_{1:n}}(x)$ is Schur-concave with respect to $\underline{\lambda}$ .\\
%We also observe that $\dfrac{\partial}{\partial \lambda_i}(\overline{F}_{X_{1:n}}(x)) <0$ $\forall ~ \alpha$, i.e., $\overline{F}_{X_{1:n}}(x)$ is decreasing over each $\lambda_i$. \\
Finally we obtain the following two cases:
\begin{enumerate}[(i)]
	\item for $0<\alpha <1$, $\underline{\lambda} \prec^{m} \underline{\mu} \Rightarrow \overline{F}_{X_{1:n}}(x) \leq \overline{F}_{Y_{1:n}}(x)$,
	\item for $\alpha >1$, $\underline{\lambda} \prec^{m} \underline{\mu} \Rightarrow \overline{F}_{X_{1:n}}(x) \geq \overline{F}_{Y_{1:n}}(x)$.
\end{enumerate}
Hence the result follows.\\
\end{proof}
A realization of the above result can be observed with the help of the following example.
\begin{example}
Consider two series systems with $4$ components each, such that the parameter vectors are taken as $\underline{\lambda}$ $=(0.8,1.2,1.3,1.9)$ and $\underline{\mu}$ $=(0.5,0.7,1.5,2.5)$, $\underline{\lambda} \prec^{m} \underline{\mu}$, $\beta=2$. Certainly theorem 1 holds true. But we observe that the above result cannot be extended to hazard rate ordering. The functional values of $\dfrac{\overline{F}_{Y_{1:4}}(x)}{\overline{F}_{X_{1:4}}(x)} = f_1(x)$ (correct upto 3 decimal places) are observed as $f_{1}(0.4)= 1.032$, $f_1(0.6)=1.051$, $f_1(1.2) = 1.033$, $f_1(1.4) = 1.008 $ when $\alpha = 0.7$ and $f_{1}(0.4)= 0.975$, $f_1(0.6)=0.943$, $f_1(1.2) = 0.948$, $f_1(1.4) = 0.987 $ when $\alpha = 1.5$. Hence $\dfrac{\overline{F}_{Y_{1:4}}(x)}{\overline{F}_{X_{1:4}}(x)}$ is not monotone. 
	%\begin{figure}[h]
	%	\begin{subfigure}{.5\textwidth}
	%		\centering
	%		\includegraphics[scale=0.35]{plot1.eps} 
	%		\caption{$\overline{F}_{Y_{1:n}}-\overline{F}_{X_{1:n}} >~ 0$}
	%	\end{subfigure}
	%	\begin{subfigure}{.5\textwidth}
	%		\centering
	%		\includegraphics[scale=0.35]{plot2.eps}
	%		\caption{$\overline{F}_{Y_{1:n}}-\overline{F}_{X_{1:n}} <~ 0$}
	%	\end{subfigure}
	%\end{figure}	
	
%	\begin{figure}[h]
%		\centering
%		\begin{tabular}{cc}
%			\includegraphics[scale=0.3]{Fig1.eps} 
%			&
%			\includegraphics[scale=0.3]{Fig2.eps}\\
%			a. \textbf{Fig 1} shows that the difference  $\overline{F}_{Y_{1:4}}-\overline{F}_{X_{1:4}} >~ 0$ when $\alpha=0.7$&
%			b. \textbf{Fig 2} shows that the difference $\overline{F}_{Y_{1:4}}-\overline{F}_{X_{1:4}} <~ 0 $ when $\alpha=1.5$
%		\end{tabular}
%	\end{figure}
\end{example}

%It has also been observed that the above result cannot be extended to hazard rate ordering under same conditions. For instance we observe the ratio $\dfrac{\overline{F}_{Y_{1:4}}}{\overline{F}_{X_{1:4}}}$ for the above example.
%
%\begin{figure}[h]
%	\centering
%	\begin{tabular}{cc}
%		\includegraphics[scale=0.3]{Fig3.eps} 
%		&
%		\includegraphics[scale=0.3]{Fig4.eps}\\
%		
%		c. \textbf{Fig 3} shows the plot  $\dfrac{\overline{F}_{Y_{1:4}}}{\overline{F}_{X_{1:4}}}$ for $\alpha=0.7$ &
%		d. \textbf{Fig 4} shows the plot $\dfrac{\overline{F}_{Y_{1:4}}}{\overline{F}_{X_{1:4}}}$ for $\alpha=1.5$
%	\end{tabular}
%\end{figure}

The next result shows the usual stochastic ordering between two maximum order statistics when the parameters $\alpha$, $\beta$ are constant but only the parameter $\lambda$ is varying.
\begin{theorem}
\normalfont	Let $X_1,X_2,\ldots,X_n$ be a set of n-independent random variables where each $X_i \sim F(x;\alpha,\beta,\lambda_i)$, $\alpha >0,\beta >0,\lambda_i >0$ for $i=1,2,\ldots, n$. Another set $Y_1,Y_2,\ldots, Y_n$ be n-independent random variables such that $Y_i \sim F(x;\alpha,\beta,\mu_i)$ for $i=1,2,\ldots, n$.  Then
	$$\underline{\lambda} \prec^m \underline{\mu} \Rightarrow X_{n:n} \leq_{st} Y_{n:n},$$
	where $\underline{\lambda}=(\lambda_1,\lambda_2,\ldots, \lambda_n)$ and $\underline{\mu}=(\mu_1,\mu_2,\ldots,\mu_n)$.
\end{theorem}
\begin{proof}
The distribution function corresponding to the maximum order statistic is
\begin{equation} \label{2}
F_{X_{n:n}}(x)=\prod_{k=1}^n\left(1-e^{\lambda_k(1-e^{x^{\beta}})}\right)^{\alpha}.
\end{equation}
$F_{X_{n:n}}(x)$ is symmetric with respect to the parameter vector $\underline{\lambda}=(\lambda_1,\lambda_2,\ldots, \lambda_n)$.
%Differentiating \eqref{2} partially with respect to the parameter $\lambda_i$, we obtain
%$$\dfrac{\partial}{\partial \lambda_i}(F_{X_{n:n}}(x))=-\alpha(1-e^{x^{\beta}})F_{X_{n:n}}(x)\dfrac{e^{\lambda_i(1-e^{x^{\beta}})}}{1-e^{\lambda_i(1-e^{x^{\beta}})}}.$$
Consider for $\lambda_i \neq \lambda_j$, $i \neq j$, 
\begin{align*}
\Delta_2 &= (\lambda_i-\lambda_j)\left(\dfrac{\partial}{\partial \lambda_i}(F_{X_{n:n}}(x))- \dfrac{\partial}{\partial \lambda_j}(F_{X_{n:n}}(x))\right)\\
&= \alpha (1-e^{x^{\beta}}) F_{X_{n:n}}(x) (\lambda_i-\lambda_j)(\phi_1(\lambda_i) - \phi_1(\lambda_j)),
\end{align*}
where $\phi_1(\lambda)=1-\dfrac{1}{1-e^{\lambda(1-e^{x^{\beta}})}}$. In order to determine the sign of $\Delta_2$, we evaluate the derivative of $\phi_1(\lambda)$,
$$\phi_1^{\prime}(\lambda)=\dfrac{(e^{x^{\beta}}-1)e^{\lambda(1-e^{x^{\beta}})}}{(1-e^{\lambda(1-e^{x^{\beta}})})^2}.$$
We observe that $\phi_1^{\prime}(\lambda) > 0$, i.e., $\phi_1(\lambda)$ is an increasing function of $\lambda$. Thus $\Delta_2 \leq 0$  $\Rightarrow$ $F_{X_{n:n}}$ is Schur-concave. Hence result follows.
%\begin{align*}
%\underline{\lambda} \prec^{m} \underline{\mu} & \Rightarrow F_{X_{n:n}} \geq F_{Y_{n:n}}\\
%&\Rightarrow X_{n:n} \leq_{st} Y_{n:n}.
%\end{align*}
\end{proof}
\begin{example}
Consider Example 1, the same set of 4 components forms the parallel system and theorem 2 holds true here. Reversed hazard rate ordering may exist as observed in Figure 1, though the analytical proof is not available due to rigorous calculations. 
	\begin{figure}[h]
	\centering
	\begin{tabular}{c}
		\includegraphics[scale=0.3]{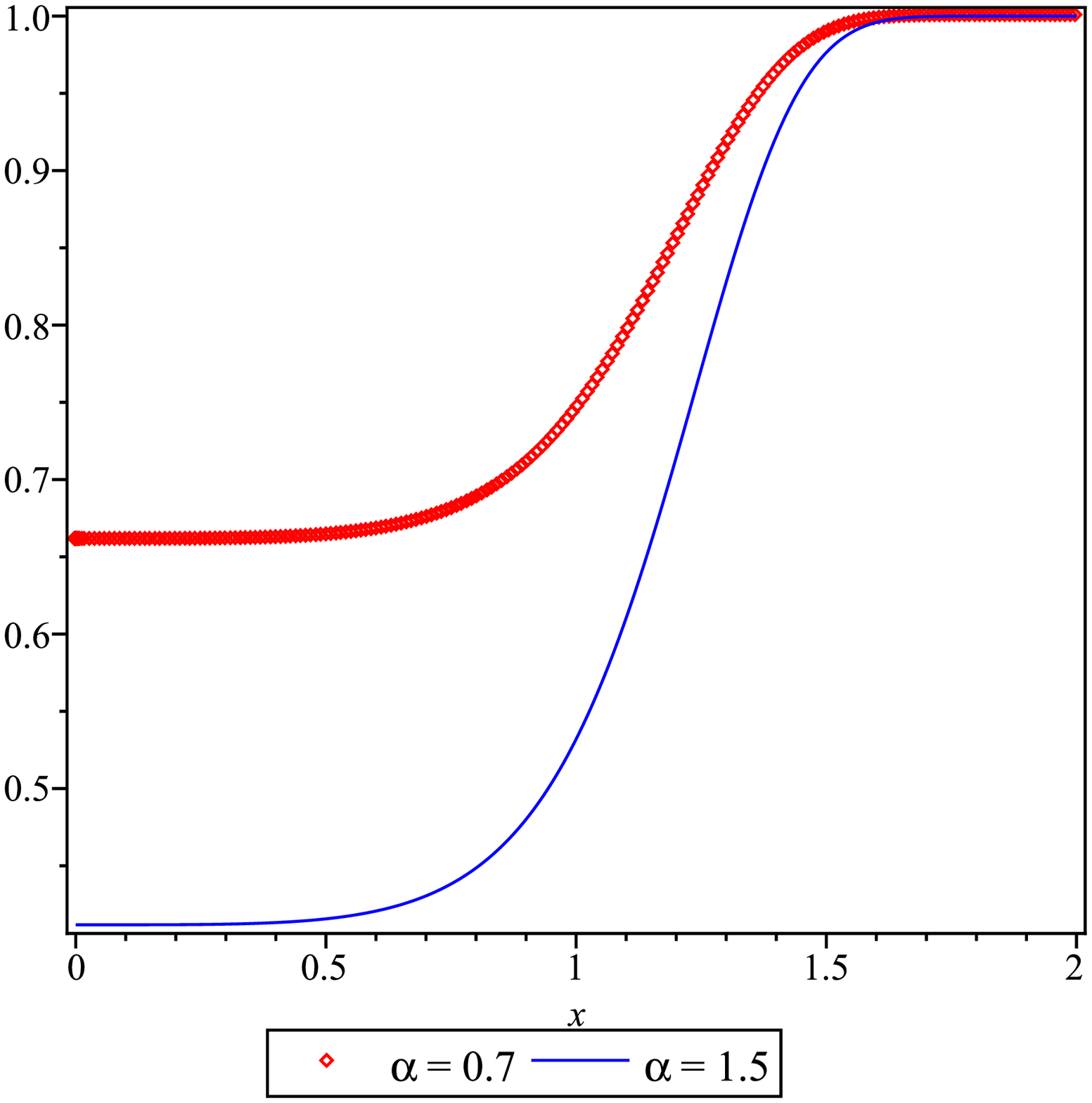}\\
		
		 \textbf{Fig 1} shows the ratio $\dfrac{F_{Y_{4:4}}}{F_{X_{4:4}}}$ for $\alpha=0.7,1.5$.
	\end{tabular}
\end{figure}
Whereas hazard rate ordering is not possible, a close look on the functional values of the ratio $\dfrac{\overline{F}_{Y_{4:4}}}{\overline{F}_{X_{4:4}}} = f_2(x)$ (say) for $\alpha=0.7$ are $f_2(1.9524) = 521403$, $f_2(1.9528) = 503289$, $f_2(1.9536) = 703257$  and $f_2(1.9524) = 558646$, $f_2(1.9528) = 539238$, $f_2(1.954) = 727174$ for $\alpha=1.5$. These values show that $\dfrac{\overline{F}_{Y_{4:4}}}{\overline{F}_{X_{4:4}}}$ is not monotone.
	\end{example}
In the next theorem, we observe usual stochastic ordering relation between two parallel system when the parameter $\beta$ is varied and all the other parameters remain constant.
\begin{theorem}
\normalfont	Consider two parallel systems consisting of n components, the components of one system corresponds to the set of n-independent random variables $X_1,X_2,\ldots,X_n$ such that $X_i \sim F(x,\alpha,\beta_i,\lambda)$ for $i=1,2,\ldots,n$. Let $Y_1,Y_2,\ldots,Y_n$ be the random variables corresponding to the components of the other system, and $Y_i \sim F(x,\alpha,\beta_i^{*},\lambda)$ for $i=1,2,\ldots,n$. Then $\lambda >1$ and  $\underline{\beta} \prec^m \underline{\beta}^{*}$ $\Rightarrow X_{n:n} \leq_{st} Y_{n:n}$, where $\underline{\beta}=(\beta_1,\beta_2,\ldots,\beta_n)$ and $\underline{\beta}^{*}=(\beta_1^{*},\beta_2^{*},\ldots,\beta_n^{*})$
\end{theorem}
\begin{proof}
 The distribution function corresponding to $X_{n:n}$ is
\begin{equation}\label{3}
F_{X_{n:n}}(x)=\prod_{k=1}^n\left(1-e^{\lambda(1-e^{x^{\beta_k}})}\right)^{\alpha}.
\end{equation}
$F_{X_{n:n}}(x)$ is symmetric with respect to the parameter vector $\underline{\beta}=(\beta_1,\beta_2,\ldots,\beta_n)$.

%Differentiating \eqref{3} partially with respect to $\beta_i$ we have
%$$\dfrac{\partial}{\partial \beta_i}F_{X_{n:n}}(x)=F_{X_{n:n}}(x)\dfrac{\alpha\lambda  x^{\beta_i}e^{\lambda(1-e^{x^{\beta_i}})+x^{\beta_i}}ln x}{1-e^{\lambda(1-e^{x^{\beta_i}})}}$$

Consider $\beta_i \neq \beta_j$, for $i \neq j$
\begin{align*}
\Delta_3&=(\beta_i-\beta_j)\left(\dfrac{\partial}{\partial \beta_i}F_{X_{n:n}}(x)-\dfrac{\partial}{\partial \beta_j}F_{X_{n:n}}(x)\right)\\
&=\alpha\lambda \ln x F_{X_{n:n}}(x)(\beta_i-\beta_j)(\phi_2(x^{\beta_i})-\phi_2(x^{\beta_j})),
\end{align*}
where $\phi_2(t)=\dfrac{te^{\lambda(1-e^t)+t}}{1-e^{\lambda(1-e^t)}}$. Computing $\phi_2^{\prime}(t)$, we observe
$$\phi_2^{\prime}(t)=-\dfrac{e^{\lambda(1-e^t)+t}((t+1)e^{\lambda(1-e^t)}+\lambda te^t-t-1)}{(e^{\lambda(1-e^t)}-1)^2}.$$
Let $g(t)=(t+1)e^{\lambda(1-e^t)}+\lambda t e^t -t-1$ and $g(0)=0$.\\
Also, $g^{\prime}(t)=(1-e^{\lambda(1-e^t)})(\lambda(t+1)e^t-1)$. In order to determine the sign of $g^{\prime}(t)$, let $h(t)=\lambda(t+1)e^t-1$ and $h(0)=\lambda-1$. Again differentiating $h(t)$, we obtain
\begin{align*}
h^{\prime}(t)&=\lambda e^t(t+2)\\
&>0 ~ \forall ~ t >0,
\end{align*}
i.e., $h(t)$ is an increasing function of $t$.
Thus, $h(t) >0$ when $\lambda>1$, and this implies $g(t)$ is an increasing function of $t$, i.e., $g(t)>0$. Finally we conclude that $\Delta_3 \leq 0$, or in other words $F_{X_{n:n}}$ is Schur concave with respect to the parameter vector $\underline{\beta}$. Hence result follows. %$\underline{\beta} \prec^{m} \underline{\beta}^{*}$ $\Rightarrow F_{X_{n:n}}(x) \geq F_{Y_{n:n}}(x)$, i.e., $X_{n:n} \leq_{st} Y_{n:n}$ for $\lambda>1$.\\
\end{proof}
We can realize the above theorem with the help of the following example.
\begin{example}
	Let us consider a  4 component parallel system, such that the random variables corresponding to each component follows the distribution function $F(x, 0.6, \beta_i,2)$, as described in theorem 3. Let $\underline{\beta}= (0.4,0.9,2,7.5) $ and $\underline{\beta}^{*}=(0.2,1,1.9,7.7)$, here $\underline{\beta} \prec^{m} \underline{\beta}^{*}$. Usual stochastic ordering exists in this case as discussed in Theorem 3. \\
	Also this result cannot be extended further to hazard rate order or reversed hazard rate order as the functional value for $f_3(x) = \dfrac{F_{Y_{4:4}}(x)}{F_{X_{4:4}}(x)}$ and $f_4(x) = \dfrac{\overline{F}_{Y_{4:4}}(x)}{\overline{F}_{X_{4:4}}(x)}$ are $f_3(0.085)= 0.0512$, $f_3(0.086) = 0.0488$, $f_3(0.087) = 0.0513$ and $f_4(9.6) = 1.453 \times 10^{6}$, $f_4(9.8) = 2.729 \times 10^{6}$, $f_4(9.9) = 2.646 \times 10^{6}$ respectively. This shows that both the functions $\dfrac{F_{Y_{4:4}}(x)}{F_{X_{4:4}}(x)}$ and $\dfrac{\overline{F}_{Y_{4:4}}(x)}{\overline{F}_{X_{4:4}}(x)}$ are not monotone.
% We shall now plot the difference between $F_{X_{4:4}}(x) - F_{Y_{4:4}}(x)$, 
%	
%	\begin{figure}[h]
%		\centering
%		\begin{tabular}{c}
%			\includegraphics[scale=0.3]{Fig31.eps} \\
%			\textbf{Fig 11} shows $F_{X_{4:4}}(x) - F_{Y_{4:4}}(x) > 0$.
%		\end{tabular}
%	\end{figure}
%
%also we can observe the plot for $\dfrac{F_{Y_{4:4}}(x)}{F_{X_{4:4}}(x)}$ and  $\dfrac{\overline{F}_{Y_{4:4}}(x)}{\overline{F}_{X_{4:4}}(x)}$ here 
%
%\begin{figure}[h]
%	\centering
%	\begin{tabular}{cc}
%		\includegraphics[scale=0.3]{Fig311.eps} 
%		&
%		\includegraphics[scale=0.3]{Fig312.eps}\\
%		
%		c. \textbf{Fig 12} shows the plot  $\dfrac{F_{Y_{4:4}}(x)}{F_{X_{4:4}}(x)}$ is not monotone &
%		d. \textbf{Fig 13} shows the plot $\dfrac{\overline{F}_{Y_{4:4}}(x)}{\overline{F}_{X_{4:4}}(x)}$ is increasing.
%	\end{tabular}
%\end{figure}
%
%Thus the reversed hazard rate ordering is not possible but the hazard rate ordering may exist. Now if the same set of components form a series system, in that case the plot for $\overline{F}_{X_{1:4}}(x)-  \overline{F}_{Y_{1:4}}(x)$ is 
%
%\begin{figure}[h]
%	\centering
%	\begin{tabular}{c}
%		\includegraphics[scale=0.3]{Fig313.eps} \\
%		\textbf{Fig 14} shows $\overline{F}_{X_{1:4}}(x)- \overline{F}_{Y_{1:4}}(x) > 0$.
%	\end{tabular}
%\end{figure}
%This shows that there exists a possibility of usual stochastic ordering, but the analytical proof is quite complicated, and this can be considered as a future problem.

	\end{example}
The next result describes the usual stochastic order relation for the minimum ordered statistic with the parameters $\beta,\lambda$ being constant and only the parameter $\alpha$ varies.
\begin{theorem}
\normalfont	Let $X_1,X_2,\ldots,X_n$ be the n-independent random variables corresponding to the components of a series system such that each $X_i$ follows the survival function $\overline{F}(x,\alpha_i,\beta,\lambda)$ for $i=1,2,\ldots,n$. Let $Y_1,Y_2,\ldots,Y_n$ be the set of n-independent random variables corresponding to another series system where the random variables $Y_i \sim \overline{F}(x,\alpha_i^{*},\beta,\lambda)$ for $i=1,2,\ldots,n$. Then as $\underline{\alpha} \prec^m \underline{\alpha}^{*} \Rightarrow X_{1:n} \leq_{st} Y_{1:n}$, where $\underline{\alpha}=(\alpha_1,\alpha_2,\ldots,\alpha_n)$, $\underline{\alpha}^{*}=(\alpha_1^{*},\alpha_2^{*},\ldots,\alpha_n^{*})$
\end{theorem}
\begin{proof}
The survival function of $X_{1:n}$ is 
\begin{equation} \label{alpha}
\overline{F}_{X_{1:n}}(x)= \prod_{k=1}^n\left(1-(1-e^{\lambda(1-e^{x^{\beta}})})^{\alpha_k}\right)
\end{equation}
$\overline{F}_{X_{1:n}}(x)$ is symmetric with respect to the parameter vector $\underline{\alpha}=(\alpha_1,\alpha_2,\ldots,\alpha_n)$, 
%now differentiating \eqref{alpha} partially with respect to $\alpha_i$ we obtain
%$$\dfrac{\partial}{\partial \alpha_i}\overline{F}_{X_{1:n}}(x)=-\overline{F}_{X_{1:n}}(x) \ln(1-e^{\lambda(1-e^{x^{\beta}})})\dfrac{(1-e^{\lambda(1-e^{x^{\beta}})})^{\alpha_i}}{1-(1-e^{\lambda(1-e^{x^{\beta}})})^{\alpha_i}}$$
For $\alpha_i \neq \alpha_j$, $i \neq j$, consider 
\begin{align*}
\Delta_4 &= (\alpha_i-\alpha_j)\left(\dfrac{\partial}{\partial \alpha_i}\overline{F}_{X_{1:n}}(x) - \dfrac{\partial}{\partial \alpha_j}\overline{F}_{X_{1:n}}(x)\right)\\
&= (\alpha_i-\alpha_j)\overline{F}_{X_{1:n}}(x) \ln(1-e^{\lambda(1-e^{x^{\beta}})}) (\psi_2(\alpha_i)-\psi_2(\alpha_j))
\end{align*}
where $\psi_2(\alpha)=1-\dfrac{1}{1-(1-e^{\lambda(1-e^{x^{\beta}})})^{\alpha}}$.\\
It has been observed that $\psi_2^{\prime}(\alpha) = -\ln(1-e^{\lambda(1-e^{x^{\beta}})})\dfrac{(1-e^{\lambda(1-e^{x^{\beta}})})^{\alpha}}{(1-(1-e^{\lambda(1-e^{x^{\beta}})})^{\alpha})^2}$, and the quantity is positive.
Therefore, $\Delta_4 \leq 0$, i.e., $\overline{F}_{X_{1:n}}(x)$ is Schur concave. Hence, $\underline{\alpha} \prec \underline{\alpha}^{*}$ $\Rightarrow \overline{F}_{X_{1:n}}(x) \leq \overline{F}_{Y_{1:n}}(x)$ and the result follows.
\end{proof}
\begin{example}
	Consider two series systems each having 3 components. The survival function of each of these components is $\overline{F}(x,\alpha_i,\beta, \lambda)$ for one system (random variables corresponding to each component is $X_1, X_2, X_3$) and  $\overline{F}(x,\alpha_i^{*},\beta, \lambda)$ for the other system (corresponding random variables are $Y_1, Y_2, Y_3$). The parameters, $\beta=3$, $\lambda= 2$, and $\underline{\alpha}= (0.2,1,2.4)$, $\underline{\alpha}^{*}= (0.4,1,2.2)$. The plot for the difference of their survival function, $\overline{F}_{X_{1:3}}(x) - \overline{F}_{Y_{1:3}}(x)$.
	\begin{figure}[h]
		\centering
		\begin{tabular}{c}
			\includegraphics[scale=0.3]{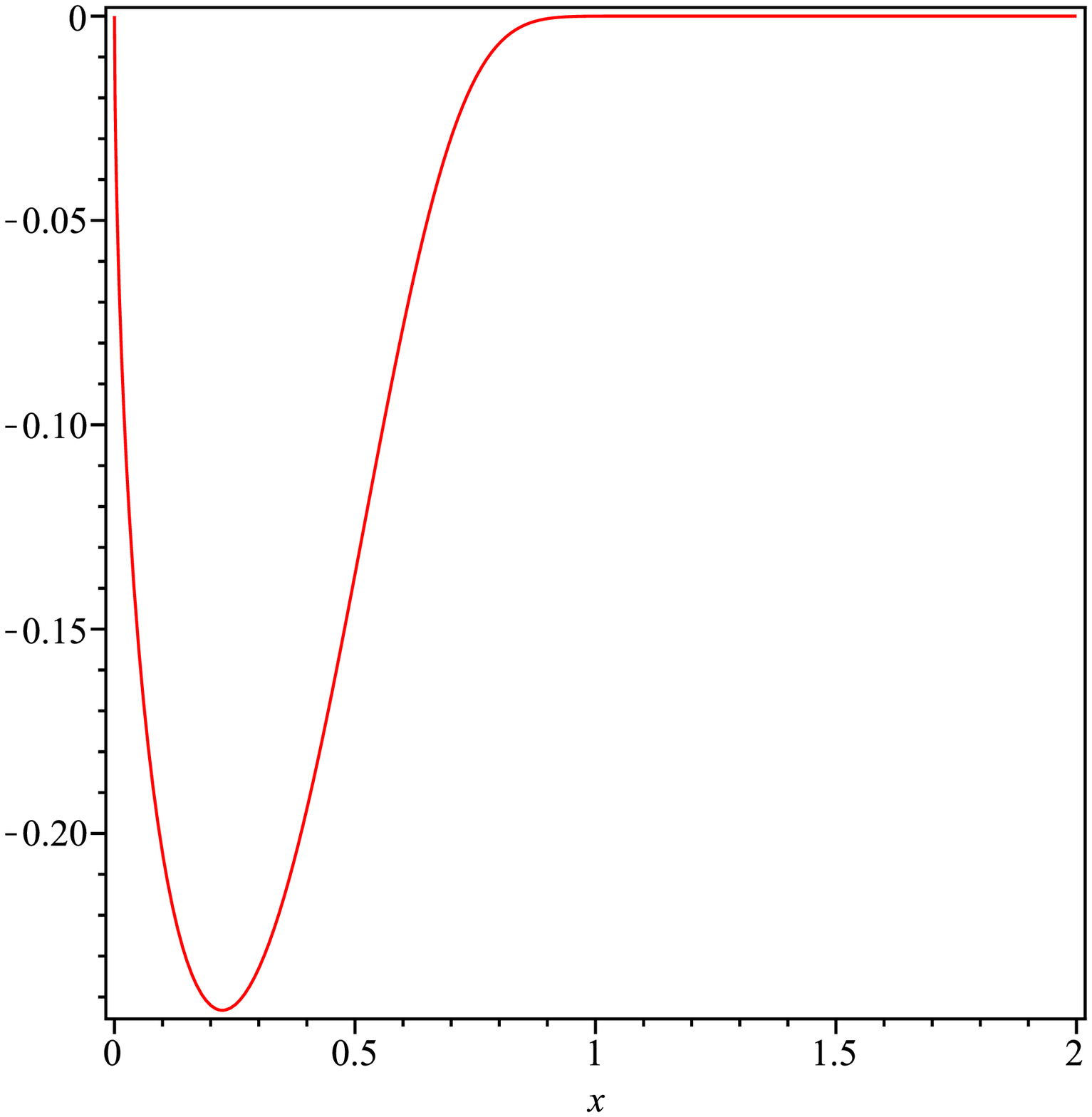} \\
			
			 \textbf{Fig 2} The figure depicts the graph of $\overline{F}_{X_{1:3}}-\overline{F}_{Y_{1:3}}$ for varying $\alpha$
		\end{tabular}
	\end{figure}

The above graph shows that the difference is negative and hence there exists an usual stochastic ordering between $X_{1:3}$ and $Y_{1:3}$. In other words theorem 4 holds true for a 3 component system. 
%	\begin{figure}[h]
%		\centering
%		\begin{subfigure}{.5\textwidth}
%			\includegraphics[scale=0.3]{plot6.eps} 
%			\caption{$\overline{F}_{X_{1:3}}-\overline{F}_{Y_{1:3}}$ for varying $\alpha$}
%		\end{subfigure}
%	\end{figure}
\end{example}

\begin{theorem}
	The random variables are same as mentioned in Theorem 4, then $X_{n:n} \leq_{lr} Y_{n:n}$ as $\displaystyle\sum_{k=1}^n \alpha_k \leq \displaystyle\sum_{k=1}^n \alpha^{*}_k$.
\end{theorem}
\begin{proof}
The distribution function of the parallel system represented as $X_{n:n}$ is
$$F_{X_{n:n}}(x)=\left(1-e^{\lambda(1-e^{x^b})}\right)^{\displaystyle\sum_{k=1}^n\alpha_k}$$
and the corresponding probability density function is
$$f_{X_{n:n}}(x)=\displaystyle\sum_{k=1}^n\alpha_k\left(1-e^{\lambda(1-e^{x^b})}\right)^{(\displaystyle\sum_{k=1}^n\alpha_k-1)}\beta\lambda x^{\beta -1}e^{\lambda(1-e^{x^{\beta}})+x^{\beta}}.$$
It is enough to prove that the ratio $\dfrac{f_{Y_{n:n}}(x)}{f_{X_{n:n}}(x)}$ is increasing in $x$. And 
$$\dfrac{f_{Y_{n:n}}(x)}{f_{X_{n:n}}(x)}= \dfrac{\displaystyle\sum_{k=1}^n \alpha^{*}_k}{\displaystyle\sum_{k=1}^n \alpha_k} \left(1-e^{\lambda(1-e^{x^{\beta}})}\right)^{(\displaystyle\sum_{k=1}^n \alpha^{*}_k - \displaystyle\sum_{k=1}^n \alpha_k)}.$$
Differentiating with respect to $x$, we find that 
\begin{align*}
\dfrac{d}{dx}\left(\dfrac{f_{Y_{n:n}}(x)}{f_{X_{n:n}}(x)}\right)&= \beta \lambda x^{\beta-1} e^{(\lambda(1-e^{x^{\beta}})+x^{\beta})} \dfrac{\displaystyle\sum_{k=1}^n \alpha^{*}_k}{\displaystyle\sum_{k=1}^n \alpha_k}\left(\displaystyle\sum_{k=1}^n \alpha^{*}_k - \displaystyle\sum_{k=1}^n \alpha_k\right) \left(1-e^{\lambda(1-e^{x^{\beta}})}\right)^{(\displaystyle\sum_{k=1}^n \alpha^{*}_k - \displaystyle\sum_{k=1}^n \alpha_k-1)}\\
& \geq 0 \text{ for } \displaystyle\sum_{k=1}^n \alpha_k \leq \displaystyle\sum_{k=1}^n \alpha^{*}_k.
\end{align*}
Hence the result follows.
\end{proof}

\section*{Conclusion}
The results discussed in this paper includes the usual stochastic ordering between $X_{1:n}$ and $Y_{1:n}$ (two series system) when the parameter $\lambda$ is varied,
or when only the parameter $\alpha$ has been varied. Whereas for two parallel system, usual stochastic ordering exists when only the parameter $\lambda$ is varied, or only the parameter $\beta$ is varied. Also likelihood ratio ordering exists when the parameter $\alpha$ varies. 

\section*{Decleration}

Funding: The first author is thankful to Indian Institute of Technology Kharagpur for research assistance-ship.\\
Conflicts of interest/Competing interests: The authors declare that they have no conflict of interest.\\
Availability of data and material: Not Applicable\\
Code availability: Not Applicable.


\end{document}